\documentclass[11pt]{amsart}

\usepackage[english]{babel}
\usepackage[T1]{fontenc}
\usepackage[utf8]{inputenc} 

\usepackage{amssymb,latexsym,comment}
\usepackage{mathtools, setspace}
\usepackage{enumerate}  
\usepackage{a4wide}    

\usepackage[colorlinks=true,citecolor=blue,linkcolor=blue]{hyperref}
\usepackage{cleveref}
\usepackage{latexsym,amsmath,amssymb}
\usepackage{accents}
\usepackage{color}

\usepackage[colorinlistoftodos,prependcaption,textsize=tiny]{todonotes}

\title[Calder\'on-Zygmund for nonlocal systems]{Calder\'on-Zygmund theory for strongly coupled linear system of nonlocal equations with H\"older-regular coefficient}

\author[T. Mengesha]{Tadele Mengesha}
\address[Tadele Mengesha]{Department of Mathematics, The University of Tennessee, Knoxville, 204 Ayres Hall, 1403 Circle Drive
Knoxville, TN, 37996.
}
\email{mengesha@utk.edu}

\author[A. Schikorra]{Armin Schikorra}
\address[Armin Schikorra]{Department of Mathematics,
University of Pittsburgh,
301 Thackeray Hall,
Pittsburgh, PA 15260, USA}
\email{armin@pitt.edu}

\author[A. Seesanea]{Adisak Seesanea}
\address[Adisak Seesanea]{Sirindhorn International Institute of Technology, Thammasat University, Pathum Thani 12120, Thailand}
\email{adisak.see@siit.tu.ac.th}

\author[S. Yeepo]{Sasikarn Yeepo}
\address[Sasikarn Yeepo]{Sirindhorn International Institute of Technology, Thammasat University, Pathum Thani 12120, Thailand}
\email{sasikarn.y@g.siit.tu.ac.th}

plus 6pt minus 12pt
\def\eps{\varepsilon}

\newcommand{\subsubset}{\subset\subset}

\newtheorem{theorem}{Theorem}

\newtheorem{lemma}[theorem]{Lemma}

\newtheorem{proposition}[theorem]{Proposition}

\theoremstyle{definition}

\def\lip{{\rm Lip\,}}

\def\supp{{\rm supp\,}}

\newcommand{\R}{\mathbb{R}}

\newcommand{\brac}[1]{\left (#1 \right )}

\newcommand{\barint}{
\rule[.036in]{.12in}{.009in}\kern-.16in \displaystyle\int }

\newcommand{\barcal}{\mbox{$ \rule[.036in]{.11in}{.007in}\kern-.128in\int $}}

\def\mvint_#1{\mathchoice
          {\mathop{\vrule width 6pt height 3 pt depth -2.5pt
                  \kern -8pt \intop}\nolimits_{\kern -3pt #1}}%
          {\mathop{\vrule width 5pt height 3 pt depth -2.6pt
                  \kern -6pt \intop}\nolimits_{#1}}%
          {\mathop{\vrule width 5pt height 3 pt depth -2.6pt
                  \kern -6pt \intop}\nolimits_{#1}}%
          {\mathop{\vrule width 5pt height 3 pt depth -2.6pt
                  \kern -6pt \intop}\nolimits_{#1}}}

\numberwithin{theorem}{section} \numberwithin{equation}{section}

\newcommand{\lap}{\Delta }

\newcommand{\aleq}{\lesssim}
\newcommand{\ageq}{\gtrsim}
\newcommand{\aeq}{\asymp}
\newcommand{\Rz}{\mathcal{R}}
\newcommand{\laps}[1]{(-\lap) ^{\frac{#1}{2}}}

\newcommand{\lapms}[1]{I^{#1}}

\begin{document}

\begin{abstract}

We extend the Calder\'on-Zygmund theory for nonlocal equations to 
strongly coupled system of linear nonlocal equations $\mathcal{L}^{s}_{A} u  = f$, where the operator $\mathcal{L}^{s}_{A}$ is formally given by 
\[
\mathcal{L}^s_{A}u = \int_{\mathbb{R}^n}\frac{A(x, y)}{\vert x-y\vert ^{n+2s}} \frac{(x-y)\otimes (x-y)}{\vert x-y\vert ^2}(u(x)-u(y))dy.
\]
For $0 < s < 1$ and $A:\mathbb{R}^{n} \times \mathbb{R}^{n} \to \mathbb{R}$ taken to be symmetric and serving as
a variable coefficient for the operator, the system under consideration is the fractional version of the classical Navier-Lam\'e linearized elasticity system.  The study of the coupled system of nonlocal equations is motivated by its appearance in nonlocal mechanics, primarily  in peridynamics.  Our regularity result states that if $A(\cdot, y)$ is uniformly Holder continuous and $\inf_{x\in \mathbb{R}^n}A(x, x) > 0$, then for $f\in L^{p}_{loc},$ for $p\geq 2$, the solution vector $u\in H^{2s-\delta,p}_{loc}$ for some $\delta\in (0, s)$.

\end{abstract}
\everymath{\displaystyle}

\setstretch{1.2}
\maketitle

\section{Introduction}
\subsection{Motivation}
The goal of this work is to obtain Sobolev regularity estimates for solutions of the strongly coupled system of linear nonlocal equations $\mathcal{L}^{s}_{A} u  = f$, where the operator $\mathcal{L}^{s}_{A}$ is formally given by 
\[
\mathcal{L}^s_{A}u = \int_{\mathbb{R}^n}\frac{A(x, y)}{|x-y|^{n+2s}} \frac{(x-y)\otimes (x-y)}{|x-y|^2}(u(x)-u(y))dy.
\]
Here we take $n\geq 1$, $0<s<1$, and $A:\mathbb{R}^{n}\times \mathbb{R}^n \to \mathbb{R}$ is taken to be symmetric and serves as a variable coefficient for the operator $\mathcal{L}^s_{A}$.  For vectors $a=(a_1, \cdots, a_n)$ and $b=(b_1, \cdots, b_n)$ in $\mathbb{R}^n$, the tensor product $a\otimes b$ is the rank one matrix with its $(ij)^{th}$ entry being $a_ib_j$.  

Coupled systems of linear nonlocal equations of the above type appear in applications. In fact, the operator $\mathcal{L}^s_{A}$ is related to the the bond-based linearized peridynamic equation \cite{Silling2001, silling2010linearized}. To briefly describe where the operator comes from, consider a  heterogeneous elastic solid occupying the domain $\Omega$ in $\mathbb{R}^n$, $n=1, 2,$ or $3$, that is linearly deforming when subjected to an external force field $f$. 
In the framework of the peridynamic model, a bounded domain hosting an elastic material is conceptualized as a sophisticated mass-spring system. 
Here, any pair of points $x$ and $y$ within the material is considered to interact through the bond vector $x-y$. 
When external load $f$ is applied, the material undergoes a deformation, mapping a point $x$ in the domain to the point $x + u(x)\in \mathbb{R}^{n}$, where the vector field $u$ represents the
displacement field. Adhering to the principles of uniform small strain theory \cite{silling2010linearized}, the strain of the bond $x-y$ is given by the
nonlocal linearized strain \[
s[u](x, y)= \frac{u(x)-u(y)}{|x-y|}\cdot \frac{x-y}{|x-y|}.
\]
 The linearized bond-based peridynamic static model relates the displacement field $u$ and the external load $f$ by the equation \cite{Du-Gun-Leh-Zhou, Mengesha-Du-2014}
\[
\int_{\mathbb{R}^n} \mathcal{C}(s[u](x, y), x, y) dy = f(x),\quad x\in \Omega.
\] 
where the vector-valued pairwise force density function $\mathcal{C}$ is given by  
\[\mathcal{C}(s[u](x, y), x, y) = A(x, y)\, \rho(x-y)\, s[u](x, y)\, \frac{x-y}{|x-y|}. \]
In the above $A(x,y)$ serves as a 'spring constant' for the bond joining $x$ and $y$ and the function $\rho$ is the interaction kernel that is radial and describes the force strength
between material points. After noting that 
\[
s[u](x, y) \frac{x-y}{|x-y|} = \frac{(x-y)\otimes (x-y)}{|x-y|^2}\frac{u(x)-u(y)}{|x-y|}
\]
then $\mathcal{L}^{s}_{A}$ is precisely the linearized bond-based peridynamic operator corresponding to the kernel of interaction $\rho(x-y) = \frac{1}{|x-y|^{n+2(s-1)}}$. 
\subsection{Statement of the main result}
Our interest is to  address the question of regularity of  solutions $u$ to $\mathcal{L}^s_A u = f$ relative to the data $f$. To that end, we require the coefficient $A$ to satisfy some continuity and boundedness assumptions. First, we say $A$ satisfies a uniform H\"older continuity assumption if for some $\alpha \in (0, 1)$ and $\Lambda>0$, 
\begin{equation}\label{alpha-holder}
    \sup_{z\in \mathbb{R}^n} |A(z,x) - A(z, y)| \leq \Lambda |x-y|^\alpha.
\end{equation}
Given $\lambda, \Lambda >0$ and $\alpha \in (0, 1)$, we define the coefficient class
\[
\mathcal{A}(\alpha, \lambda, \Lambda) = \left \{A: A(x, y) = A(y,x), \,\inf_{x\mathbb{R}^n} K(x, x) > \lambda,\quad \|A\|_{L^{\infty}} \leq \frac{1}{\lambda},\, \text{and satisfies \eqref{alpha-holder}} \right \}.
\]
Given $A\in L^{\infty}(\mathbb{R}^{n}\times \mathbb{R}^{n})$ and $u\in L^{1}_{loc}(\mathbb{R}^{n}, \mathbb{R}^{n})$
we understand $\mathcal{L}^s_{A}u$ as a distribution defined as 
\[
\langle\mathcal{L}^s_{A}u,\varphi\rangle := \int_{\mathbb{R}^n}\int_{\mathbb{R}^n}\frac{A(x, y)}{|x-y|^{n+2s}}(u(y)-u(x))\cdot\frac{(y-x)}{|y-x|}  (\varphi(y)-\varphi(x))\cdot\frac{(y-x)}{|y-x|}dydx
\]
for all $\varphi \in C_c^\infty(\R^n,\R^n).$ 
Moreover, if  $u\in H^{s}(\mathbb{R}^{n}, \mathbb{R}^{n})$, then from the above definition, $\mathcal{L}^{s}_{A}u \in H^{-s}(\mathbb{R}^{n}, \mathbb{R}^{n})$ with the estimate that 
\[
\|\mathcal{L}^{s}_{A}u\|_{H^{-s}} \leq {1\over \lambda}\|u\|_{H^{s}}.
\]
Now given an open set $\Omega\subset \mathbb{R}^{n}$ and $f\in H^{-s}(\mathbb{R}^n, \mathbb{R}^{n}), $ a vector field $u\in H^{s}(\mathbb{R}^{n}, \mathbb{R}^{n})$ is a solution to $\mathcal{L}^{s}_{A}u = f$ in $\Omega$ if 
\begin{equation}\label{intro:equation}
\langle\mathcal{L}^{s}_{A} u, \varphi \rangle = \langle f, \varphi\rangle, 
\,\,\text{for all $\varphi\in C_c^\infty(\Omega,\R^n)$}.
\end{equation}
In the event, $A=1$, then operator agrees with the integral operator defined as  
\[
    (-\mathring{\Delta})^{s}u (x) := {\emph p.v.} \int_{\mathbb{R}^{n}}{1\over |z|^{n+2s}} \left({z\otimes z\over |z|^2}\right)(u(x) - u(x+z)) dz
\]
where the integral converges in the sense of principal value for smooth vector fields.  Notice that if $\mathcal{F}$ is the Fourier transform, then for vector fields $u$ in the Schwarz space $\mathcal{S}(\mathbb{R}^n, \mathbb{R}^n)$, we have 
\begin{equation}\label{Fourier-vector-Lap}
\mathcal{F}((-\mathring{\Delta})^{s}u) = (2\pi |\xi|)^{2s}(\ell_1\mathbb{I} + \ell_{2} {\xi\otimes\xi\over |\xi|^2}) \mathcal{F}(u)
\end{equation}
for some positive constants $\ell_1$ and $\ell_2$ depending only on $n$ and $s.$ 
As a consequence, as shown in \cite{Mengesh-Scott2019} for any $\tau>0$ and $f\in L^{p}(\mathbb{R}^{n}, \mathbb{R}^{n})$ with  $1<p<\infty,$ then the solution $u$ to 
\[
(-\mathring{\Delta})^{s}u + \tau u = f
\]
lives in $H^{2s, p}(\mathbb{R}^n, \mathbb{R}^n)$. For the nonlocal equation of variable coefficient \eqref{intro:equation},  we would like to obtain a Sobolev regularity of the above type for solutions  in the event that the right hand side $f$ has additional regularity.
We begin by noting that for some $\lambda$ and $\Lambda$, $A\in \mathcal{A}(\alpha, \lambda, \Lambda)$, and $f\in H^{-s}(\mathbb{R}^n, \mathbb{R}^{n})$, a solution to \eqref{intro:equation} exists under some volumetric condition on $u$. Indeed, a minimizer of the energy 
\[
E(u) = {1\over 2} \langle\mathcal{L}^{s}_{A} u, u \rangle - \langle f, u\rangle
\]
over the space $V=\{u\in H^{s}(\mathbb{R}^n,\mathbb{R}^{n}): u=0 \quad \text{on $\mathbb{R}^{n}\setminus \Omega$}  \}$ will satisfy the equation \eqref{intro:equation}. The existence of a minimizer for the quadratic functional $E$ over $V$, with a possible sign changing $A\in \mathcal{A}(\alpha, \lambda, \Lambda)$ will be shown later.
As has been demonstrated in \cite{Mengesha-Du-2014}, with a proper multiplicative constant $c(s,n)$, in terms of the nonlocality parameter $s$, the operator $c(s,n)\mathcal{L}^{s}_{A}u$ that corresponds to $A(x,y) = {1/2}(a(x) + a(y))$ will converge in an appropriate sense to the Lam\'e differential operator 
\[
L_0 u (x) = \text{div}(a(x)\nabla u) +2 \nabla\left( a(x) \text{div} \,u(x)\right).
\]
This operator is strongly elliptic in the sense of Legendre-Hadamard but not uniformly elliptic.  
One can then view \eqref{intro:equation} as a fractional analogue of the classical Navier-Lam\'e system of linearized elasticity equation. 

The main result of the paper is the following interior regularity estimate which is  the version of the regularity result proved in \cite{MSY21} for the coupled system of nonlocal equations under discussion.

\begin{theorem}\label{th:main}
Let  $s\in(0, 1)$ and $s \leq t<\min\{ 2s, 1\}$.  Let $\Omega\subset \mathbb{R}^{n}$ be an open bounded set. 
If for $2\leq q<\infty$, $f_1,f_2 \in L^q(\Omega,\mathbb{R}^n)\cap L^2(\R^n,\mathbb{R}^n)$, and  $u\in  H^{s}(\mathbb{R}^n,\mathbb{R}^n)$  is  a distributional solution of $\mathcal{L}^{s}_{A}u = (-\Delta)^{2s-t\over 2} f_1 + f_2$ in $\Omega$, in the sense that, 
\[
\langle\mathcal{L}^{s}_{A}u, \varphi\rangle=\int_{\R^{n}}  \langle f_1, \laps{2s-t} \varphi\, \rangle dx  + \int_{\R^{n}} \langle f_2, \varphi\,\rangle dx \,\quad \forall \varphi \in C_c^\infty(\Omega;\R^n), 
\]
with $\mathcal{L}^{s}_{A}$ corresponding to $A\in \mathcal{A}(\alpha, \lambda, \Lambda)$ for some given $\alpha\in (0, 1)$ and $\lambda, \Lambda>0$, then we have $\laps{t} u \in L^q_{loc}(\Omega,\mathbb{R}^n)$ and for any $\Omega' \subsubset \Omega$  we have 
\[
 \|\laps{t} u\|_{L^q(\Omega')} \leq C \brac{\|u\|_{W^{s,2}(\mathbb{R}^{n})} + \sum_{i=1}^2 \|f_i\|_{L^q(\Omega)}+\|f_i\|_{L^2(\R^n)} }.
\]
The constant $C$ depends only $s$,$t$,$q$,$\alpha$,$\lambda$,$\Lambda$,$\Omega$, and $\Omega'$.  
\end{theorem}
The proof of the theorem parallels the approach used in \cite{MSY21}. Namely, we compare the operator $\mathcal{L}_{A}^{s}$ with the simpler operator $\bar{\mathfrak{L}}_{A_{D}}^{s_1,s_2}$, where $s_1 + s_2 = 2s$, and is defined as, for $u\in H^{s}(\mathbb{R}^n ,\mathbb{R}^n)$ and $\varphi\in C_c^{\infty}(\mathbb{R}^n, \mathbb{R}^n)$, 
\begin{equation}\label{fractional-Lame}
\langle \bar{\mathfrak{L}}_{A_D}^{s_1,s_2} u, \varphi\rangle = \int_{\mathbb{R}^{n}} A_{D}(z)\left\langle (c_{1}\mathbb{I} + c_{2} \mathcal{R}\otimes \mathcal{R})(-\Delta)^{s_1\over 2}u(z), (-\Delta)^{\frac{s_2}{2}}\varphi (z) \right\rangle dz
\end{equation}
for constants $c_1$ and $c_2$ that will be determined as a function $s$ and $n$. In the above definition, the operator $\mathcal{R} = (\mathcal{R}_1, \mathcal{R}_2, \cdots, \mathcal{R}_n)$ is the vector of Riesz transforms,  
and $A_{D}(z) = A(z, z),$ the restriction of the coefficient $A$ on the diagonal. Notice that for constant coefficients the two operators $\mathcal{L}_{A}^{s}$ and $\bar{\mathfrak{L}}_{A_D}^{s,s}$ coincide. Indeed, if $A(x, y) = A,$ constant, then by using \eqref{Fourier-vector-Lap}, for vector fields in the Schwarz space 
\[
\mathcal{L}_{A}^{s}u = A  (-\mathring{\Delta})^{s}u = A(\ell_1(-\Delta)^{s\over2} u + \ell_2 (\mathbf{R}\otimes \mathbf{R})(-\Delta)^{s\over 2}u) = \bar{\mathfrak{L}}_{A}^{s,s} u
\]
with $c_1=\ell_1$ and $c_2 = \ell_2$.  We will prove an optimal regularity result for solutions of the strongly coupled equation 
\begin{equation}\label{Base-system}
\langle \bar{\mathfrak{L}}_{A_D}^{s_1,s_2} u, \varphi\rangle = \langle g, \varphi\rangle,\quad \forall \varphi\in C_c^{\infty}(\mathbb{R}^n, \mathbb{R}^n)
\end{equation}
and use those solutions as approximations of the solution to the original system of equations. The mechanism we accomplish this is via perturbation argument where we show that  the difference operator 
\[
\mathcal{D}_{s,t} u = \mathcal{L}_{A}^{s}u - \bar{\mathfrak{L}}_{A_D}^{s,t} u
\]
can be understood as a lower order term in the event that $A$ is H\"older continuous. 

While our work studies solutions to strongly coupled linear nonlocal pdes, there has been a number of results in the literature that studied the regularity of solutions to scalar nonlocal pdes. To name a few, optimal local regularity results are obtained in \cite{BiccariWarmaZuazua+2017+387+409} for weak solutions to the Dirichlet problem associated with the fractional Laplacian. Similar results are obtained for the fractional heat equation in \cite{10.1007/978-3-030-00874-1_2,GRUBB20182634}. Almost optimal regularity results are obtained in \cite{Matteo2017} for weak solutions to nonlocal equations with  H\"older regular coefficients. Optimal Sobolev regularity are proved in \cite{DONG20121166} for strong solutions to nonlocal equations with translation invariant coefficients ($A(x, y) = A(x-y)$). For equations with less regular coefficients, higher integrability and higher differentiability results are obtained in \cite{Nowak2023,NOWAK2020111730} for nonlocal equations with variable coefficients that have small mean oscillations. See also \cite{Fall1,Fall2} for related results. For elliptic, measurable, and bounded coefficients, solutions to nonlocal pdes are proved in \cite{KMS15} to have a self-improvement property where  higher integrability and higher differentiability are obtained without any smoothness assumption on the coefficients, see also \cite{S16}. Similar results are also verified in \cite{Scott-Mengesha-2022,byun2023regularity} for solutions to nonlocal double phase problems. For a concise description of the results of  the above mentioned manuscripts, we refer to \cite{MSY21}.

The paper is organized as follows. In the next section we estimate $\langle\mathcal{D}_{s,t} u, \varphi\rangle$ in terms of the Riesz potential $I^{s} = (-\Delta)^{-\frac{s}{2}}$.  In Section \ref{sec-Base-system}, we will develop the optimal regularity result for a solution of equation \eqref{Base-system}. In Section \ref{sec-regularity-proof}, we prove the main result of the paper by using an iterative argument making use of the commutator estimate we prove in Section \ref{s:commies} and the optimal regularity result obtained in Section \ref{sec-Base-system}.
\subsection{Notation and some preliminaries}
We now fix notations and convention we will use throughout the paper. We will also discuss some preliminary results we need in the sequel. 

We begin by noting that domains of integrals are always open sets and we use the symbol $\subsubset$ to say compactly contained, e.g. $\Omega_1 \subsubset \Omega_2$ if $\overline{\Omega_1}$ is compact and $\overline{\Omega_1} \subset \Omega_2$.
 Constants change from line to line, and unless it is important we may not detail their dependence on various parameters. We will make frequent use of $\aleq$, $\ageq$ and $\aeq$, which denotes inequalities with multiplicative constants (depending on non-essential data). For example we say $A \aleq B$ if for some constant $C > 0$ we have $A \leq C B$. We will use the angle bracket $\langle\cdot, \cdot\rangle$ to represent the standard inner product or the duality pairing depending on the context. 

Our arguments below make use of the various definition and properties of fractional Laplacian operators, and accompanying Sobolev spaces, see  \cite{DNPV12,G19}, or monographs \cite{S02} for more on fractional operators. We will make use of 
Sobolev inequalities and various embedding 
that can be found in \cite{RS96}.

To that end, for $s \in (0,2)$ the fractional Laplacian $\laps{s}$ is,  defined via the Fourier transform,
\[
\laps{s}u = \mathcal{F}^{-1}(2\pi |\xi|^{s} \hat{u} )
\]
where the Fourier transform is defined as $\mathcal{F}(u) (\xi) = \hat{u}(\xi) = \int_{\mathbb{R}^{n} } e^{-2\pi\imath x\cdot \xi} u(x) dx$. It also has a useful integral representation and for any vector field $u$ in the Schwartz class
\[
 \laps{s} u(x) = c_{s,n} \emph{p.v.}\int_{\R^n} \frac{u(x)-u(y)}{|x-y|^{n+s}}\, dy, 
\]
where \emph{p.v} stands for the  \emph{principal value}, whose mentioning we will suppress. The inverse operator of the fractional Laplacian is the Riesz potential whose integral 
representation is 
\[
 (-\lap)^{-\frac{s}{2}} v(x) \equiv \lapms{s} v(x) = c\int_{\R^n} \frac{v(y)}{|x-y|^{n-s}} \, dy
\]
for a vector field $v$ in the Schwartz class. Sobolev inequalities needed for this paper are proved in \cite[Proposition 2.1]{MSY21} (see also \cite{alma996965690102311}) and we summarize them as follows. 

\begin{lemma}\label{S-I}
\begin{enumerate}
\item[(a)] If $sp < n$, then there exists a constant $C=C(s,p,n)>0$ such that 
\begin{equation}\label{eq:sob:glob}
 \|\lapms{s} v\|_{L^{\frac{np}{n-sp}}(\R^n)} \leq C\, \|v\|_{L^p(\R^n)}\quad \text{for any $v \in L^p(\R^n,\mathbb{R}^n)$}.
\end{equation}
\end{enumerate}
In addition, if $\Omega\subset\R^{n}$ is bounded, then corresponding to any $q \in [1,\frac{np}{n-sp}]$, there is a constant $C=C(s,p,n, \Omega)>0$ such that 
\begin{equation}\label{eq:sob:loc1}
 \|\lapms{s} v\|_{L^q(\Omega)} \leq C\, \|v\|_{L^p(\R^n)} \quad \text{for any $v \in L^p(\R^n,\mathbb{R}^n)$}.
\end{equation}
\begin{enumerate}

\item [(b)]If $sp \geq n$ and $\Omega\subset\R^{n}$ is bounded domain,  then for any $q \in [1,\infty)$, and $r \in [1,\frac{n}{s})$, there exists a constant $C=C(s,p,n, \Omega)>0$ such that for any $v\in L^p(\R^n,\mathbb{R}^n)$
\begin{equation}\label{eq:sob:loc2}
 \|\lapms{s} v\|_{L^q(\Omega)} \leq C\, \brac{\|v\|_{L^p(\R^n)}+\|v\|_{L^r(\R^n)}}.
\end{equation}
\end{enumerate}
\end{lemma}
The above Sobolev estimates together with the relationship between the fractional Laplacian and the Riesz potentials yield the following result that is also stated and proved in \cite[Proposition 2.4]{MSY21}. We state it here in a slightly different way to suit our setting. 

\begin{lemma}[\cite{MSY21}]\label{Prop2.4MSY}
Suppose that $\eta_1, \eta_2\in C^{\infty}_{c}(\mathbb{R}^{n})$, and $\eta_2 = 1$ in the neighborhood of the support of $\eta_1$. Then for any $\psi \in C_c^{\infty}(\mathbb{R}^{n})$ such that $\text{Supp}(\psi)\subset \{x: \eta_1(x) = 1\},$ and any $q, p\in(1,\infty)$ and $\tau\in (0, 2) $ we have
\begin{equation}\label{prop2.4-parta}
\|(1-\eta_2)\laps{\tau} ((1-\eta_1)\lapms{\tau} \psi)\|_{L^{q}(\mathbb{R}^{n})} \leq C \|\psi\|_{L^{p}(\mathbb{R}^{n})}. 
\end{equation}
Moreover, if $r > {np\over n+\tau p} > 1$ for $\tau <1$, then for any bounded set $\Omega\subset \mathbb{R}^n$, there exists a constant $C(\Omega)$ such that
\begin{equation}\label{prop2.4-partb}
\|\laps{\tau} ((1-\eta_1)\lapms{\tau} \psi)\|_{L^{r'}(\Omega)} \leq C \|\psi\|_{L^{p'}(\mathbb{R}^{n})}. 
\end{equation}
In either case the constant $C$ may also depend on $r$, $q$, $\tau$, $p$, $n$, and on $\eta_1, \eta_2$,  but not on $\psi$. 
\end{lemma}
Notice that because of the strict inclusion of the support of $1-\eta_2$ into the support of $1-\eta_1$, the inequality \eqref{prop2.4-parta} holds for any $p, q\in (1, \infty)$. The way it is written here, the inequality is slightly different from part a) of \cite[Proposition 2.4]{MSY21} but the same proof can be repeated for the proof of \eqref{prop2.4-parta}. 

 We also mention the dual definition of $\laps{s}$ operator. Indeed, for vector fields $u$ and $v$ in the Schwartz class, the $L^2$-inner product of $ \laps{s} u(x)$ and  $v(x)$  can be represented as, for $s \in (0,2)$,
 \begin{equation}\label{eq:fraclap}
 \int_{\R^n} \laps{s} u(x) \cdot \, v(x) dx = \int_{\R^n} \int_{\R^n} \frac{(u(y)-v(x))\cdot(v(y)-v(x))}{|x-y|^{n+s}}\, dx\, dy.
 \end{equation}
 The proof can be found \cite[Proposition 2.36.]{ArminPhD} or \cite{DNPV12}.

 The Riesz transform, $\Rz = (\Rz_1,\ldots,\Rz_n) := \nabla \lapms{1}$, plays a central role in this work. First, $\Rz$ has  the Fourier symbol $c \imath \frac{\xi}{|\xi|}$, and can also be 
 represented as 
 \[
  \Rz f(x) = \int_{\R^n} \frac{\frac{x-y}{|x-y|} }{|x-y|^{n}}\, f(y)\, dy.
 \]
Second, we will use the fact that they are Calder\'on-Zygmund operators and for $1<p<\infty$, there exists a constant $C=C(n,p)>0$ such that 
\[
\|  \Rz f\|_{L^{p}} \leq C\|f\|_{L^{p}}, \quad\text{for all $f\in L^{p}$}. 
\]
Finally, we state and prove existence of a solution to the nonlocal system \eqref{intro:equation}. We note that members of the coefficient class $\mathcal{A}(\alpha, \lambda,\Lambda)$ can be negative off diagonal. Indeed, as indicated in \cite{MSY21}, the coefficient $A(x, y) = {2\lambda + |x|^\alpha + |y|^\alpha\over \lambda + |x|^\alpha + |y|^\alpha} + 10^{6}(\sin x+\sin y)
{|x-y|^\alpha \over 1+|x-y|^\alpha }$ belongs to $\mathcal{A}(\alpha, \lambda,\Lambda)$ and yet can be negative off diagonal. 

\begin{proposition}
   Suppose that $\Omega\subset \mathbb{R}^n$ is an open bounded set with Lipschitz boundary and that $A\in \mathcal{A}(\alpha, \lambda, \Lambda)$ for $\alpha\in(0, 1)$. Then for any, $s\in (0, 1)$, there exists $C>0$ so that if ${\lambda \over \Lambda} > C$ and $f\in H^{-s}(\R^n, \R^n)$, there exists $u\in H^{s}(\R^n, \R^n)$ such that 
   \[
   \begin{cases}
       \mathcal{L}_{A} u &= f\quad \text{in $\Omega$}\\
       u&=0\quad\text{in $\R^n\setminus \Omega$}, 
   \end{cases}
   \]
in the sense that $u=0$ in $\R^n\setminus \Omega$ and  
$
\langle  \mathcal{L}_{A} u, \varphi\rangle = \langle f, \varphi \rangle,
$ for all $\varphi\in C_c^{\infty}(\Omega,\R^n)$. Moreover the solution minimizes the energy 
\[
E_{A}(v) = {1\over 2} \langle\mathcal{L}^{s}_{A} v, v \rangle - \langle f, v\rangle
\]
over the space $V=\{v\in H^{s}(\mathbb{R}^n,\mathbb{R}^{n}): v=0 \ \text{on $\mathbb{R}^{n}\setminus \Omega$}  \}$.
\end{proposition}
\begin{proof}
It suffices to show that the energy $E_A$ has a minimizer. To that end, first notice that if $A\in \mathcal{A}(\alpha, \lambda, \Lambda)$, then for any $\beta\in (0, \alpha)$, we have
$A\in \mathcal{A}(\beta, \lambda, \tilde{\Lambda}_{\beta})$ for some $\tilde{\Lambda}_\beta>0$. This follows from the estimate that for any $\beta\in (0, \alpha)$ and any $x, y, z\in \mathbb{R}^n$ 
\[
|A(z, x) - A(z, y)| = |A(z, x) - A(z, y)|^{\beta\over \alpha} |A(z, x) - A(z, y)|^{1-{\beta\over \alpha}} \leq \left|{2\over \lambda}\right|^{1-{\beta\over \alpha}} \Lambda^{\beta\over \alpha}|x-y|^\beta. 
\]
Thus, without loss of generality we may assume that $0<\alpha<2s.$ Next, we show that 
$E_A: V \to \mathbb{R}$ coercive.  Indeed, for any $v\in V$,
\[
\begin{split}
E_A(v) &=  {1\over 2}\int_{\R^n}\int_{\R^n} A_{D}(x, x){|(v(y)-v(x))\cdot{(y-x)\over|y-x|}|^{2}\over |y-x|^{n+2s}} dy dx+ E_{A-A_{D}}(v)\\
&\geq {\lambda \over 2} \int_{\R^n}\int_{\R^n}{|(v(y)-v(x))\cdot{(y-x)\over|y-x|}|^{2}\over |y-x|^{n+2s}} dy dx-{\Lambda\over 2}\int_{\R^n}\int_{\R^n}{|(v(y)-v(x))\cdot{(y-x)\over|y-x|}|^{2}\over |y-x|^{n+2(s-{\alpha\over 2})}} dy dx\\
&\quad\quad-\|f\|_{H^{-s}}\|v\|_{H^{s}(\R^n)}.
\end{split}
\]
Using the fractional Korn's inequality in \cite{SM2019-Korn} which proves for any $s\in (0, 1)$ and $v\in H^{s}(\R^n, R^n)$, 
$
\int_{\R^n}\int_{\R^n}{|(v(y)-v(x))\cdot{(y-x)\over|y-x|}|^{2}\over |y-x|^{n+2s}} dy dx \approx |v|^{2}_{H^{s}(\R^n)}
$, we conclude that 
\[
E_A(v) \ageq {\lambda\over 2} |v|^{2}_{H^{s}(\R^n)} - {\Lambda\over 2} |v|^{2}_{H^{s-{\alpha\over 2}}(\R^n)} -\|f\|_{H^{-s}}\|v\|_{H^{s}(\R^n)}.
\]
Applying the fractional Poincar\'e inequality on $V$, we have $|v|^{2}_{H^{s-{\alpha\over 2}}(\R^n)}\geq C  \|v\|_{L^{2}(\R^n)}^{2}$, we have 
\[
{E_A(v) \over |v|^{2}_{H^{s}(\mathbb{R}^{n})}} \geq {\lambda \over 2} - {C\Lambda\over 2} {\|v\|_{L^{2}(\R^n)}^{2} \over |v|^{2}_{H^{s}(\mathbb{R}^{n})}} - {\|f\|_{H^{-s}} \over |v|_{H^{s}(\mathbb{R}^{n})}}.
\]
Another application of fractional Poincar\'e inequality implies that there exists $\tilde{C}>0$ such  for ${\lambda \over \Lambda} > \tilde{C}$ we have 
\[
{E_A(v) \over |v|^{2}_{H^{s}(\mathbb{R}^{n})}} > {\Lambda \over 4},\quad \text{ for all $v$ such that $|v|_{H^{s}(\mathbb{R}^{n})} \to \infty$}. 
\]
We can now apply standard direct method of calculus variations to demonstrate existence of a minimizing vector field. 
\end{proof}

\section{Commutator estimates}\label{s:commies}
In this section, we obtain estimates for the quantity $\mathcal{D}_{s_1,s_2}u$ defined in the previous section. To be precise, for $s\in (0, 1)$ such that $s_1 + s_2 = 2s$, $u\in H^{s}(\mathbb{R}^n, \R^{n})$, and $\varphi\in C_c(\mathbb{R}^{n}, \mathbb{R}^{n})$, we recall that 
\begin{equation}\label{1st-div}
\mathcal{D}_{s_1,s_2}(u,\varphi) = \langle \mathcal{L}^{s}_{A} u, \varphi\rangle -\langle \bar{\mathfrak{L}}_{A_D}^{s_1,s_2} u,\varphi \rangle.  
\end{equation}
Before we begin estimating this difference, let us first find a different characterization of the operator $\mathcal{L}^{s}_{A}u$ that uses Riesz potentials. To that end, 
for any $x,y \in \R^n$ and $\varphi\in C_c(\mathbb{R}^{n}, \mathbb{R}^{n})$, 
\begin{equation}\label{fn-via-RP}
\begin{split}
\varphi(x)-\varphi(y) &= 
c_2\brac{\lapms{s_2}\laps{s_2}\varphi(x)-\lapms{s_2}\laps{s_2}\varphi(y)}\\
&= c_2\int_{\R^n} \laps{s_2} \varphi(z_2)\brac{|x-z_2|^{s_2-n}-|y-z_2|^{s_2-n}}\, dz_2, 
\end{split}
\end{equation} 
for a constant $c_2$ that depends only on $s_2$ and $n$. 
This identity remains valid for $\varphi\in H^{s_2}(\mathbb{R}^{n}, \mathbb{R}^{n})$. Similarly, for any  $0\leq \eps<s_2$, we can write 
\[
\begin{split}
\varphi(x)-\varphi(y)&= c_2(\epsilon)\brac{\lapms{s_2-\eps}\laps{s_2-\eps}\varphi(x)-\lapms{s_2-\eps}\laps{s_2-\eps}\varphi(y)}\\
&= c_2(\epsilon)\int_{\R^n} \laps{s_2-\eps} \varphi(z_2) \brac{|x-z_2|^{s_2-\eps-n}-|y-z_2|^{s_2-\eps-n}}\, dz_2
\end{split}
\]
where $c_2(\epsilon)$ depends on $\eps$, in addition to $s_2$ and $n$. We note that $c_2(\epsilon)  = c_2 >0.$ 
Now we plug in $\varphi(x) - \varphi(y)$ and $u(x) - u(y)$ in 
\[
\langle \mathcal{L}^{s}_{A}u,\varphi\rangle =\int_{\mathbb{R}^{n}} \int_{\R^{n}} {A(x,y)\over |x-y|^{n+ 2s}} \left\langle {(x-y)\otimes (x-y)\over |x-y|^2} (u(x) - u(y)), \varphi(x) - \varphi(y)\right\rangle dx dy  
\]
and apply Fubini's theorem to obtain the expression that for any $\epsilon\in [0, s_2)$
\begin{equation}\label{Def-L(s,A)-interms-Rp}
\begin{split}
\langle \mathcal{L}^{s}_{A}u,\varphi\rangle = &\int_{\R^n}\int_{\R^n} \big\langle\mathbb{K}^{\epsilon}_{A}(z_1,z_2) \laps{s_1} u(z_1), \, \laps{s_2-\epsilon} \varphi(z_2)\big\rangle\, dz_1\, dz_2, 
 \end{split}
\end{equation}
where for any set function $B(x, y)$, and $0\leq \epsilon <s_2,$
\begin{equation}\label{defn-K}
\mathbb{K}^{\epsilon}_{B}(z_1, z_2) = \int_{\R^n}\int_{\R^n} \frac{B(x, y)}{|x-y|^{n+2s}} \kappa_{s_1}^{s_2 -\epsilon}(x, y, z_1, z_2) {(x-y)\otimes (x-y)\over |x-y|^2}dx dy 
\end{equation} with, $c(\epsilon) = c_1 \cdot  c_2(\epsilon)$,
\begin{equation}\label{defn-kappa}
\kappa_{s_1}^{s_2 - \epsilon }(x, y, z_1, z_2)= c(\epsilon)(|x -z_1|^{s_1-n}-|y-z_1|^{s_1-n})\, (|x-z_2|^{s_2-\epsilon -n}-|y-z_2|^{s_2-\epsilon-n}). 
\end{equation}
See \cite[Lemma 3.6.]{MSY21} for a rigorous justification of the above calculations.   With this at hand, for $\epsilon\in [0, s_2)$,  we introduce an intermediate operator $ A_{D} \mathcal{L}^{s}_{1}$ given by 
\[
\langle A_{D} \mathcal{L}^{s}_{1} u, \varphi\rangle := \int_{\R^n}\int_{\R^n} A_{D}(z_1)\big\langle\mathbb{K}^{0}_{1}(z_1,z_2) \laps{s_1} u(z_1), \, \laps{s_2} \varphi(z_2)\big\rangle\, dz_1\, dz_2
\]
where $\mathbb{K}^{0}_{1}$ is as given in \eqref{defn-K} with $\epsilon=0$ and $B=1$.

Now, for a given $\epsilon\in [0, s_2)$, we first write the difference $\mathcal{D}_{s_1,s_2}(u,\varphi)$ defined in \eqref{1st-div} as
\[
\begin{split}
\mathcal{D}_{s_1,s_2}(u,\varphi) &= \langle\mathcal{L}^{s}_{A} u -\left(A_{D}\mathcal{L}^{s}_{1}\right)u, \varphi\rangle + \langle\left(A_{D}\mathcal{L}^{s}_{1}\right)u, \varphi\rangle - \langle\bar{\mathfrak{L}}_{A_D}^{s_1,s_2} u, \varphi\rangle
\\
&= \mathcal{D}_{1}^{s}(u,\varphi) + \mathcal{D}^{s_1,s_2}_{2}(u,\varphi).
\end{split}
\]
We then have that 
\begin{equation}\label{2nd-div}
\langle\mathcal{L}^{s}_{A} u, \varphi\rangle  = \langle\bar{\mathfrak{L}}_{A_D}^{s_1,s_2} u, \varphi\rangle + \mathcal{D}_{1}^{s}(u,\varphi) + \mathcal{D}^{s_1,s_2}_{2}(u,\varphi). 
\end{equation}
The next two propositions estimate the last two terms of \eqref{2nd-div}. 
First, we estimate $\mathcal{D}_{1}^{s}(u,\varphi)=\langle\mathcal{L}^{s}_{A} u -\left(A_{D}\mathcal{L}^{s}_{1}\right)u, \varphi\rangle$.  
\begin{proposition}\label{pr:prop11}
Let $s \in (0,1)$ with $s_1 + s_2 = 2s$, $\alpha \in (0,1)$ and $\Lambda > 0$. Then there exist constants $\sigma_0 \in (0,\alpha]$ and a constant $c>0$ such that 
 the decomposition \eqref{2nd-div} holds and  for any $A\in \mathcal{A}(\alpha,\Lambda) = \{A:\mathbb{R}^{n}\times\mathbb{R}^{n}\to \mathbb{R}:  |A(x, y)|\leq \Lambda, \,\text{\eqref{alpha-holder} holds} \}$, any $\sigma \in (0,\sigma_0)$, and any $\eps \in (0,\frac{\sigma}{4})$, we have 
\[
|\mathcal{D}_{1}^{s}(u,\varphi)| \aleq  \int_{\R^n}\lapms{\sigma-\eps} |\laps{s_1}u|(x)\, |\laps{s_2-\eps}\varphi|(x)\, dx,
\]
and 
\[
|\mathcal{D}_{1}^{s}(u,\varphi)| \aleq  \int_{\R^n}\lapms{\sigma-\eps} |\laps{s_1-\eps}u|(x)\, |\laps{s_2}\varphi|(x)\, dx
\]
for all $u \in H^{s_1,p}(\R^n, \R^n)$ and $\varphi \in C^\infty_c(\R^n,\R^n)$.
\end{proposition}
\begin{proof}
First notice from the definition of the operator $A_D \mathcal{L}^{s}_1$ and integrating in the $z_2$ variable that for any $\epsilon\in[0, s_2)$
\[
\langle A_{D} \mathcal{L}^{s}_{1} u, \varphi\rangle = \int_{\R^n}\int_{\R^n} A_{D}(z_1)\big\langle\mathbb{K}^{\epsilon}_{1}(z_1,z_2) \laps{s_1} u(z_1), \, \laps{s_2-\epsilon} \varphi(z_2)\big\rangle\, dz_1\, dz_2, 
\]
and so using the \eqref{Def-L(s,A)-interms-Rp} and the definition of $A_{D} \mathcal{L}^{s}_{1}$, we have that 
\begin{equation}\label{defn-D-wo-Epsilon}
\mathcal{D}_{1}^{s}(u,\varphi) = \int_{\R^n}\int_{\R^n} \big\langle\mathbb{M}^{\epsilon}(z_1, z_2) \laps{s_1} u(z_1), \, \laps{s_2-\epsilon} \varphi(z_2)\big\rangle\, dz_1\, dz_2
\end{equation}
where 
\[
\begin{split}
\mathbb{M}^{\epsilon}(z_1, z_2) &= \mathbb{K}^{\epsilon}_{A}(z_1,z_2) - A_{D}(z_1)\mathbb{K}^{\epsilon}_{1}(z_1,z_2)\\
&=\int_{\R^n}\int_{\R^n} \frac{A(x, y) -A_{D}(z_1)}{|x-y|^{n+2s}} \kappa_{s_1}^{s_2 -\epsilon}(x, y, z_1, z_2) {(x-y)\otimes (x-y)\over |x-y|^2}dx dy 
\end{split}
\] and $\kappa_{s_1}^{s_2 -\epsilon}$ as defined in \eqref{defn-kappa}. It then follows that 
\[
|\mathbb{M}^{\epsilon}(z_1, z_2)| \leq \int_{\mathbb{R}^{n}}\int_{\mathbb{R}^n} |A(x, y) -  A(z_1, z_2)|{| \kappa_{s_1}^{s_2 -\epsilon}(x, y, z_1, z_2)| \over |x-y|^{n+2s}}| dx dy, 
\]
and as a consequence, 
\[
|\mathcal{D}_{1}^{s}(u,\varphi)|\leq \int_{\R^n}\int_{\R^n}|\mathbb{M}^{\epsilon}(z_1, z_2)| |\laps{s_1} u(z_1)||\laps{s_2 -\epsilon} \varphi(z_2)| dz_1 dz_2.
\]
We observe that the upper bound of $|\mathbb{M}^{\epsilon}(z_1, z_2)|$ is exactly the quantity that appear in \cite[Lemma 3.5]{MSY21}, and so  for $\sigma>0$ small enough, the inequality 
\[
|\mathcal{D}_{1}^{s}(u,\varphi)| \aleq  \int_{\R^n}\lapms{\sigma-\eps} |\laps{s_1}u|(x)\, |\laps{s_2-\eps}\varphi|(x)\, dx.
\]
follows from \cite[Theorem 3.1]{MSY21}. 
The other estimate follows the same way by reversing the role of $u$ and $\varphi$. 	
\end{proof}
Next we estimate $ \mathcal{D}^{s_1,s_2}_{2}(u,\varphi) = \langle A_{D} \mathcal{L}^{s}_{1} u, \varphi\rangle -\langle\bar{\mathfrak{L}}_{A_D}^{s_{1},s_2} u, \varphi\rangle$.  
Recall that  the operator $A_{D} \mathcal{L}^{s}_{1}$ is defined as follows: 
\[
\langle A_{D} \mathcal{L}^{s}_{1} u, \varphi\rangle = \int_{\R^n}\int_{\R^n} A_{D}(z_1)\big\langle\mathbb{K}^{0}_{1}(z_1,z_2) \laps{s_1} u(z_1), \, \laps{s_2} \varphi(z_2)\big\rangle\, dz_1\, dz_2
\]
Let us obtain a compact form of the action of  the operator $A_{D} \mathcal{L}^{s}_{1}$ on vector fields. 
Denoting 
$U := \laps{s_1} u$ and $V := \laps{s_2} \varphi$,
it follows that 
\[
\begin{split}
\langle A_{D} &\mathcal{L}^{s}_{1} u, \varphi\rangle = \int_{\R^n}\int_{\R^n} A_{D}(z_1)\big\langle\mathbb{K}^{0}_{1}(z_1,z_2) U(z_1), \, V(z_2) \big\rangle\, dz_1\, dz_2 \\
=&\int_{\R^n}\int_{\R^n} \left\langle {(x-y)\otimes (x-y)\over |x-y|^{n+2(s+1)}}  (\lapms{s_1} \brac{A_{D}U}(x)-\lapms{s_1} \brac{A_{D}U}(y)), \lapms{s_2}V(x)-\lapms{s_2}V(y)  \right\rangle dx dy\\
\end{split}
\]
Observe the following for $\gamma_1 = (n+2s-2)(n+2s)$ and $\gamma_2 = (n+2s-2)$
\begin{equation}\label{eq:xixjcomp}
\nabla^{2} \left(\frac{1}{|x-y|^{n+2s-2}}\right) = \gamma_1 {(x-y)\otimes (x-y)\over |x-y|^2} \frac{1}{|x-y|^{n+2s}} - \gamma_{2} \frac{1}{|x-y|^{n+2s}}\mathbb{I}.
\end{equation}
It then follows that 
{\small \begin{equation*}
\begin{aligned}
\langle A_{D} \mathcal{L}^{s}_{1} u, \varphi\rangle&\overset{\eqref{eq:xixjcomp}}{=} \int_{\R^n}\int_{\R^n} {\tilde{\gamma}_2 \over |x-y|^{n+2s}} (\lapms{s_1} \brac{A_{D}U}(x)-\lapms{s_1} \brac{A_{D}U}(y))\cdot (\lapms{s_2}V(x)-\lapms{s_2}V(y))  dx dy\\
&+\tilde{\gamma}_1\int_{\R^n}\int_{\R^n}\left\langle\nabla^{2} \left(\frac{1}{|x-y|^{n+2s-2}}\right) (\lapms{s_1} \brac{A_{D}U}(x)-\lapms{s_1} \brac{A_{D}U}(y)), \lapms{s_2}V(x)-\lapms{s_2}V(y)  \right\rangle
\\
=&\int_{\R^n} \tilde{\gamma}_2 A_D(z) U(z) \cdot V(z) + \tilde{\gamma}_1 \left\langle \nabla^2\lapms{2} (A_{D} U)(z), V(z) \right\rangle\,dz\\
=&\int_{\R^n} \brac{\tilde{\gamma}_2 A_D(z) U(z) + \tilde{\gamma}_1 \Rz\otimes \Rz (A_{D} U)(z)}\cdot V(z)\, dz.
\end{aligned}
\end{equation*}}
The precise value of constants $\tilde{\gamma}_1$ and $\tilde{\gamma}_2$ are  computed in \cite{Scott22} and verify that $\tilde{\gamma}_1\neq \tilde{\gamma}_2$.   For the argument to follow, the exact value is not as important, but these are the constants that appear in the regular operator $\bar{\mathfrak{L}}_{A_D}^{s_{1},s_2}. $ 
As a consequence, the expression for $ \mathcal{D}^{s_1,s_2}_{2}(u,\varphi) $ simplifies to 
\[
\begin{split}
 \mathcal{D}^{s_1,s_2}_{2}(u,\varphi) &= \tilde{\gamma}_1\int_{\R^n}  \langle \left(\Rz\otimes \Rz (A_{D} \laps{s_1} u)(z)- A_{D}  \Rz\otimes \Rz ( \laps{s_1} u)(z)\right),\ \laps{s_2} \varphi(z)\rangle dz \\
 &= \tilde{\gamma}_1\int_{\R^n} \langle[\Rz\otimes \Rz,  A_{D}]( \laps{s_1} u)(z), \laps{s_2} \varphi(z)\rangle dz
\end{split}
\] 
where we used 
the commutator notation
\[
 [T,b](f) = T(bf)-bTf.
\]
We normalize the constant and assume that $\tilde{\gamma}_1=1.$ The next proposition estimates $ \mathcal{D}^{s_1,s_2}_{2}(u,\varphi)$. 
\begin{proposition}\label{pr:prop12}
Let $s \in (0,1)$ with $s_1 + s_2 = 2s$ and $\alpha>0$ . For $u \in H^{s_1,p}(\R^n)$ and $\varphi \in C^\infty_c(\R^n)$, let 
\[
\begin{split}
 \mathcal{D}^{s_1,s_2}_{2}(u,\varphi) = \int_{\R^n} \langle[\Rz\otimes \Rz,  A_{D}]( \laps{s_1} u)(z), \laps{s_2} \varphi(z)\rangle dz. 
 \end{split}
\]
Then, there exists $c>0$ such that \eqref{2nd-div} holds and 
 for any $\eps \in (0,\alpha)$ 
 and for any $A\in \mathcal{A}(\alpha,\Lambda) = \{A:\mathbb{R}^{n}\times\mathbb{R}^{n}\to \mathbb{R}:  |A(x, y)|\leq \Lambda, \,\text{\eqref{alpha-holder} holds} \}$
we have the estimates
\[
| \mathcal{D}^{s_1,s_2}_{2}(u,\varphi)| 
	\leq  \int_{\R^n} \lapms{\alpha-\eps}|\laps{s_1-\eps}u|(z)\, |\laps{s_2}\varphi|(z)\, dz, 
\]
and
\[
| \mathcal{D}^{s_1,s_2}_{2}(u,\varphi)| 
	\leq \int_{\R^n} \lapms{\alpha-\eps}|\laps{s_1}u|(z)\, |\laps{s_2-\eps}\varphi|(z)\, dz
\]
for all $u \in H^{s_1,p}(\R^n)$ and $\varphi \in C^\infty_c(\R^n)$.
\end{proposition}
\begin{proof}
	By applying an integration by parts, we get
	{\small \[
	\begin{aligned}
	\mathcal{D}^{s_1,s_2}_{2}(u,\varphi) =& \int_{\R^n} \langle\laps{\eps}([\Rz\otimes \Rz, A_{D}] (\laps{s_1} u))(z),  \laps{s_2-\eps} \varphi(z)\rangle\, dz\\
    =& \int_{\R^n}\int_{\R^n}  \left\langle\frac{[\Rz \otimes \Rz, A_{D}] (\laps{s_1} u)(z) - [\Rz\otimes \Rz, A_{D}] (\laps{s_1} u)(\tilde{z})}{|z-\tilde{z}|^{n+\eps}}, \,\laps{s_2-\eps}\varphi(z)\,\right\rangle \,d\tilde{z}\, dz.
	\end{aligned}
	\]}
	Notice that
	\[
	\begin{split}
	[\Rz\otimes \Rz, A_{D}](\laps{s_1} u)(z)
	 =& \Rz\otimes \Rz (A_{D}\laps{s_1}u )(z) - A_{D}\Rz\otimes\Rz (\laps{s_1}u)(z)\\
	 =& \int_{\R^n} \frac{\Omega(y-z)}{|y-z|^n}(A(y,y)-A(z,z))\laps{s_1}u(y)\,dy.
	\end{split}
	\]
    where we are using the notation $\Omega(\xi)= {\xi\otimes \xi\over |\xi|^2}$. We recall  that ${\Omega(\xi)\over |\xi|^{n}}$ is the Calder\'on-Zygmund kernel for the second order matrix of Riesz transform $\Rz\otimes \Rz$, see \eqref{eq:xixjcomp}. Thus we have, 
    \[
	\begin{split}
	\mathcal{D}^{s_1,s_2}_{2}(u,\varphi) =\int_{\R^n}\int_{\R^n} \left\langle\mathbb{W}^\epsilon(y, z)\laps{s_1}u(y),  \laps{s_2-\eps}\varphi(z)\right\rangle dy dz
    \end{split}
    \]
 where  
    \[
   \mathbb{W}^\epsilon(y, z):= \int_{\R^n} \frac{1}{|z-\tilde{z}|^{n+\eps}}\brac{\frac{\Omega(y-z)}{|y-z|^n}(A(y,y)-A(z,z)) - \frac{\Omega(y-\tilde{z})}{|y-\tilde{z}|^n}(A(y,y)-A(\tilde{z},\tilde{z}))} \,d\tilde{z}.
    \]
     We claim that
    \[
       | \mathbb{W}^\epsilon(y, z)| \aleq |y-z|^{\alpha-\eps-n}.
    \]
  Assume the claim is proved for now. We then have  
	\[
	\begin{split}
	|\mathcal{D}^{s_1,s_2}_{2}(u,\varphi)| \aleq& \int_{\R^n} 
	\brac{\int_{\R^n} |z-y|^{\alpha-\eps-n}\, |\laps{s_1}u|(y)\,dy} 
	\, |\laps{s_2-\eps} \varphi|(z)\, dz\\
	=& \int_{\R^n} \lapms{\alpha-\eps}|\laps{s_1}u|(z)\, |\laps{s_2-\eps}\varphi|(z)\, dz.
	\end{split}
	\]
    Hence, we obtain the second estimate for $\mathcal{D}^{s_1,s_2}_{2}(u,\varphi)$. 
    By reversing the role of $u$ and $\varphi$, the first estimate follows the same way.

	What remains is to prove the claim. To that end, we  divide the domain into three cases. 
 
	\underline{Case 1} $|z-\tilde{z}| < \frac{1}{2}|y-z|$ or $|z-\tilde{z}| < \frac{1}{2}|y-\tilde{z}|$.
	
	We first consider
    \[\begin{split}
       & \Big| \frac{\Omega(y-z)}{|y-z|^n}(A(y,y)-A(z,z)) - \frac{\Omega(y-\tilde{z})}{|y-\tilde{z}|^n}(A(y,y)-A(\tilde{z},\tilde{z})) \Big|\\
        &\leq \Big| \brac{\frac{\Omega(y-z)}{|y-z|^n} -  \frac{\Omega(y-\tilde{z})}{|y-\tilde{z}|^n}} (A(y,y)-A(z,z)) \Big| + \Big|\frac{\Omega(y-\tilde{z})}{|y-\tilde{z}|^n}(A(\tilde{z},\tilde{z})-A(z,z)) \Big|.
    \end{split}
    \]
    Since in this case we have $|y-z| \aeq |y-\tilde{z}|$, we can use the application of the fundamental theorem of calculus (see \cite[Lemma 3.2]{MSY21}) to obtain
    \[
    \begin{split}
        |\mathbb{W}^\epsilon(y, z)| \aleq& \int_{\R^n} \frac{|z-\tilde{z}|}{|z-\tilde{z}|^{n+\eps}|y-z|^{n+1}}|A(y,y)-A(z,z)|  \,d\tilde{z}\\
        & + \int_{\R^n} \frac{1}{|z-\tilde{z}|^{n+\eps}|y-z|^{n}}|A(\tilde{z},\tilde{z})-A(z,z)|  \,d\tilde{z}\\
        \aleq& \int_{\R^n} \frac{|z-\tilde{z}|}{|z-\tilde{z}|^{n+\eps}|y-z|^{n+1}}|y-z|^\alpha  \,d\tilde{z} 
        + \int_{\R^n} \frac{1}{|z-\tilde{z}|^{n+\eps}|y-z|^{n}}|z-\tilde{z}|^\alpha  \,d\tilde{z}
    \end{split}
    \]
    where the second inequality above follows from the $\alpha$-H\"older continuous of $A$. Then we integrate w.r.t. $\tilde{z}$ and get   
    \[
    \begin{split}
        | \mathbb{W}^\epsilon(y, z)|  \aleq& \int_{|z-\tilde{z}| \aleq |y-z|} \frac{|z-\tilde{z}|^{1-\eps-n}}{|y-z|^{n+1-\alpha}}  \,d\tilde{z} + \int_{|z-\tilde{z}| \aleq |y-z|} \frac{|z-\tilde{z}|^{\alpha-\eps-n}}{|y-z|^n} \,d\tilde{z}\\
        \aleq& |y-z|^{1-\eps}|y-z|^{\alpha-1-n} +|y-z|^{\alpha-\eps-n}\\
        \aleq& |y-z|^{\alpha-\eps -n}.
    \end{split}
    \]
	\underline{Case 2} $|z-\tilde{z}| \ge \frac{1}{2}|y-z|$ and $|z-\tilde{z}| \ge \frac{1}{2}|y-\tilde{z}|$ and $|y-z| < |y-\tilde{z}|$.
	
	Since $A$ is $\alpha$-H\"older continuous, we have
    \[
    \begin{split}
        | \mathbb{W}^\epsilon(y, z)| \aleq& \int_{|z-\tilde{z}| \ageq |y-z|} \frac{1}{|z-\tilde{z}|^{n+\eps}}\brac{|y-z|^{\alpha-n} + |y-\tilde{z}|^{\alpha-n}} \,d\tilde{z}\\
        \leq& |y-z|^{\alpha-n} \int_{|z-\tilde{z}| \ageq |y-z|} \frac{1}{|z-\tilde{z}|^{n+\eps}} \,d\tilde{z}\\
        \aeq& |y-z|^{\alpha-\eps-n}.
    \end{split}
    \]
	\underline{Case 3} $|z-\tilde{z}| \ge \frac{1}{2}|y-z|$ and $|z-\tilde{z}| \ge \frac{1}{2}|y-\tilde{z}|$ and $|y-z| \ge |y-\tilde{z}|$.

    We have
	\[
    \begin{split}
        | \mathbb{W}^\epsilon(y, z)| \aleq& \int_{|z-\tilde{z}| \ageq |y-z|} \frac{1}{|z-\tilde{z}|^{n+\eps}}\brac{|y-z|^{\alpha-n} + |y-\tilde{z}|^{\alpha-n}} \,d\tilde{z}\\
        \aleq& |y-z|^{-\eps-n} \int_{|y-\tilde{z}| \leq |y-z|} |y-\tilde{z}|^{\alpha-n} \,d\tilde{z}\\
        \aeq& |y-z|^{\alpha-\eps-n}.
    \end{split}
    \]
	That completes the proof of the proposition. 
\end{proof}

\section{The weighted fractional Lam\'e  system}\label{sec-Base-system}
In this section, we prove an optimal regularity result for the system of equations
\begin{equation} \label{weighted-FLAME}
\bar{\mathfrak{L}}_{\bar{A}}^{t,2s-t} u = \laps{2s-t} f_1   + f_2
\end{equation}
where $\bar{A}$ is a positive, measurable function that is bounded from below and above by positive constants, and  $\bar{\mathfrak{L}}_{\bar{A}}^{t,2s-t}$ is as defined in \eqref{fractional-Lame}, and can be understood as the operator 
\[
 \laps{2s-t}\left(\bar{A}(z)[\laps{t} u(z) +c\Rz\otimes \Rz \laps{s} u(z)]\right).
\]
where after scaling in \eqref{fractional-Lame}, we assume that $c\neq 1$.  The following is an {\em a priori} regularity estimate that we will use as an iterative device to obtain the optimal regularity result for the \eqref{weighted-FLAME}. 

\begin{proposition}\label{intermediate-regular-Klaps}
Let ${s\in (0,1)}$ and $t \in (0,2s)$ such that $2s-t < 1$
Suppose that
$\bar{A}: \R^n \to \R$ is a positive, measurable, and bounded from above and below, i.e.
\[
\Lambda^{-1} \leq \bar{A}(z) \leq  \Lambda \quad \text{a.e. }z \in \R^n.
\]
Assume that for some $q \in (1,\infty)$, $\laps{t} u \in L^q(\R^n,\R^n)$ is a distributional solution to
\begin{equation} \label{tested-weighted}
\begin{split}
 & \int_{\R^n} \bar{A}(z) \left\langle \laps{t} u (z)+c\Rz\otimes \Rz \laps{t} u(z), \laps{2s-t} \varphi(z)\right \rangle\, dz \\
  &= \int_{\R^n} \langle f_1(z),\, \laps{2s-t} \varphi(z)\rangle\, dz + \int_{\R^n} \langle f_2(z),\, \varphi(z)\rangle\, dz,\quad \quad \forall \varphi \in C_c^\infty(\Omega, \R^n)
  \end{split}
\end{equation}
where $c \neq 1$.  Suppose now that $\Omega_1 \subsubset \Omega_2 \subsubset \Omega\subset \R^n$. Then 
\begin{itemize}
\item[a)] there exists $\bar{\mathfrak{q}}$ such that $\bar{\mathfrak{q}}>q > \frac{n \bar{\mathfrak{q}}}{n+(2s-t)\bar{\mathfrak{q}}}>1$ such that 
 if $f_1,f_2 \in  L^q(\R^n,\R^n) \cap L^{\bar{\mathfrak{q}}}(\Omega_2, \R^n)$, then $\laps{t} u \in L^{\bar{\mathfrak{q}}}(\Omega_1, \R^n)$ and 
 \[
\|\laps{t} u\|_{L^{\bar{\mathfrak{q}}}(\Omega_1)} \aleq \sum_{j=1}^{2}\left(\|f_j\|_{L^{\bar{\mathfrak{q}}}(\Omega_2)} + \|f_j\|_{L^q(\R^n)}\right) 
+\|\laps{t} u\|_{L^q(\R^n)}.
 \]

\item[b)] for any \text{$p>q$, and $r\in (1, p)$ such that $r>\frac{n p}{n+(2s-t)p} >1,$} 
 if $f_1,f_2 \in  L^q(\R^n,\R^n) \cap L^p(\Omega_2, \R^n)$,    then the $L^{p}$ norm of $\laps{t} u$
 can be estimated as  
{ \begin{equation}\label{eq:asdk:goalLp}
 \|\laps{t} u\|_{L^p(\Omega_1)} \aleq \sum_{j=1}^{2}\left(\|f_j\|_{L^p(\Omega_2)} + \|f_j\|_{L^q(\R^n)}\right) + \|\laps{t} u\|_{L^r(\Omega_2)} +\|\laps{t} u\|_{L^q(\R^n)}.
\end{equation}}
\end{itemize}
\end{proposition}

For $c=0$,  part b) of \Cref{intermediate-regular-Klaps} is precisely \cite[Proposition 4.1]{MSY21}. For the case when $c\neq 1$, the proof of the proposition uses arguments that parallel the proof of  \cite[Proposition 4.1]{MSY21}. Notice also that the estimate for part a) follows from part b) after we made sure $\bar{\mathfrak{q}}$ exists and taking $p=\bar{\mathfrak{q}}$ and $r=q.$ For the existence, given $q>1$, we can choose $$\bar{\mathfrak{q}}\in \left({nq\over n-(2s-t)q}, {n\over n-(2s-t)}\right).$$ 
The interval is nonempty because $q>1$. Below we will sketch the proof of part b). 
First, we state and prove the following observation, see also \cite{Scott22}.

\begin{lemma}\label{U=lameU}
Assume $c \neq 1$. Then for any $U: \R^d \to \R^d$, and $1<p<\infty,$ we have 
\begin{equation}\label{rn-inequality}
 \|U\|_{L^p(\R^d)} \aleq  \left\| U + c\,(\Rz\otimes \Rz) U\right\|_{L^p(\R^d)}.
\end{equation}
Moreover we have for any open set $\Omega_1 \subsubset \Omega_2$ and any $\tau \in [0,1]$ such that $\frac{np}{n+\tau p} > 1$
and $q \in [1,\infty)$
\[
 \|U\|_{L^p(\Omega_1)} \aleq  \left\| U + c\,(\Rz\otimes \Rz) U\right\|_{L^p(\Omega_2)}+\|U\|_{L^{\frac{np}{n+\tau p}}(\Omega_2)} + \|U\|_{L^{q}(\R^n)}
\]
With the constant depending on $\Omega_1,\Omega_2,\tau,q$.
\end{lemma}
\begin{proof}
The first estimate of lemma is known, see, e.g., \cite{Scott22}, but for the convenience of the reader we sketch the argument. After recalling that the Riesz-transform $\Rz$ is the operator with Fourier symbol $\imath \frac{\xi}{|\xi|}$, we may take the Fourier transform to obtain
\[
 \mathcal{F} \brac{U +c (\Rz\otimes\Rz) U} = C \mathbb{D}(\xi) \mathcal{F} U(\xi),
\]
where
\[
\mathbb{D}(\xi) = \mathbb{I}_n -c\frac{\xi\otimes \xi}{|\xi|^2}.
\]
We observe that $\mathbb{D}(\xi)$ is a symmetric matrix with eigenvalues $1$ (eigenspace: $\xi^\perp$ which is $d-1$-dimensional) and $1-c \neq 0$ (eigenspace ${\rm span}(\xi)$, whenever $\xi \in \R^n \setminus \{0\}$. In particular for any $\xi \in \R^n \setminus \{0\}$,  $\mathbb{D}^{-1}(\xi)$ exists, and is given by 
\[
\mathbb{D}^{-1}(\xi) = \mathbb{I}_n + {1\over 1-c} \frac{\xi\otimes \xi}{|\xi|^2}.
\]
It then follows that 
$|\mathbb{D}^{-1}(\xi)| \aleq {2-c\over 1-c}$. Then we may write 
\[
 U = \mathcal{F}^{-1} \brac{\mathbb{D}^{-1} \mathcal{F}\brac{U +c (\Rz\otimes \Rz))U}}.
\]
The claim about the $L^p$ estimate  follows from Mikhlin- or H\"ormander multiplier theorem, \cite{Mengesh-Scott2019} or \cite{Scott22} for detailed calculation.

For the second inequality, take $\eta_1,\eta_2 \in C_c^\infty(\Omega_2)$, $0\leq \eta_1, \eta_2 \leq 1,$  $\eta_1 \equiv 1$ in $\Omega_1$, $\eta_2 \equiv 1$ in $\supp \eta_1$.

Then we apply inequality \eqref{rn-inequality} to $\eta_1 U$ to obtain that 
\[
 \|U\|_{L^p(\Omega_1)} \aleq  \|\eta_1 U\|_{L^p(\mathbb{R}^{n})} \aleq \left\|\eta_1 U + c\,(\Rz\otimes \Rz) (\eta_1 U)\right\|_{L^p(\mathbb{R}^{n})}. 
\]
It then follows that 
\[
 \|U\|_{L^p(\Omega_1)} \aleq  \left\|U + c\,(\Rz\otimes \Rz) U\right\|_{L^p(\Omega_2)}+\|[\Rz \otimes \Rz,\eta_1] (U)\|_{L^p(\R^n)}, 
\]
where here we used the commutator notation
$
 [T,\eta_1](f) := T(\eta_1 f)-\eta_1 T(f).$  
To estimate the last term in the previous inequality, we use the identity $\eta_1(1-\eta_{2})\equiv 0$ to write  
\[
[\Rz \otimes \Rz,\eta_1] (U) = [\Rz \otimes \Rz,\eta_1] (\eta_2 U) + \eta_1 \Rz \otimes \Rz ((1-\eta_2)U)
\]
 and therefore 
\[
 \|[\Rz \otimes \Rz,\eta_1] (U)\|_{L^p(\R^n)} \leq
 \|[\Rz \otimes \Rz,\eta_1] (\eta_2 U)\|_{L^p(\R^n)} +
 \|\eta_1 \Rz \otimes \Rz ((1-\eta_2)U)\|_{L^p(\R^n)}.
\]
For the first term, in view of commutator estimates, say  in \cite[Theorem 6.1.]{LS18} or \cite{Ingmanns20}, for any $\tau \in [0,1]$ denoting by $\lapms{\tau}=(-\Delta)^{-\frac{\tau}{2}}$ the Riesz potential, and then using Sobolev inequality (if $\frac{np}{n+\tau p} \in (1,\infty)$)
\[
 \|[\Rz \otimes \Rz,\eta_1] (\eta_2 U)\|_{L^p(\R^n)} \aleq \|\eta_1\|_{\lip} \|\lapms{\tau} (\eta_2 U)\|_{L^{p}(\R^n)} \aleq \|\eta_2 U\|_{L^{\frac{np}{n+\tau p}}(\R^n)} \aleq \|U\|_{L^{\frac{np}{n+\tau p}}(\Omega_2)}.
\]
For the other term, we observe that for any $x\in \mathbb{R}^{n}$
\[
 |\eta_1(x) \Rz \otimes \Rz ((1-\eta_2)U)(x)| \aleq \int_{\R^n} \eta_1(x) \frac{1}{|x-y|^n} (1-\eta_2(y)) |U|(y) \aleq \kappa \ast |U|(x)
\]
where
\[
 \kappa(z) := \frac{1}{|z|^n} \chi_{|z| \ageq 1},
\]
where the constant in $\ageq$ depends on the distance of the support of $(1-\eta_2)$ to the support of $\eta_1$. We observe that $\kappa \in L^{q}(\R^n)$ for any $q \in (1,\infty]$. It then follows from Young's convolution inequality that  for any $q \in [1,\infty)$ (observing that $\kappa$ is integrable to any power)
\[
 \|\eta_1 \Rz \otimes \Rz ((1-\eta_2)U)\|_{L^p(\R^n)} \aleq \|\kappa \ast |U|\|_{L^\infty(\R^n)} \aleq \|\kappa\|_{L^{q'}(\R^n)} \|U\|_{L^q(\R^n)} \aleq \|U\|_{L^q(\R^n)} .
\]
Putting the inequalities together we complete the proof of the lemma. 
\end{proof}
We are now ready to sketch the proof of Proposition \ref{intermediate-regular-Klaps}. 
\begin{proof}[Sketch of the proof of Proposition \ref{intermediate-regular-Klaps}]
Instead of $\Omega_1$ and $\Omega_2$ we are going to prove the statement for $\Omega_1$ and $\Omega_4$ and some choice of $\Omega_2,\Omega_3$
such that $\Omega_1 \subsubset \Omega_2 \subsubset \Omega_3 \subsubset \Omega_4$. 
From the previous lemma, Lemma \ref{U=lameU}, we have
\[
\|\laps{t} v \|_{L^p(\Omega_1)} \aleq \|\laps{t} v +c\Rz\otimes \Rz \laps{t} v\|_{L^p(\Omega_2)} + \|\laps{t} u\|_{L^r(\Omega_2)} +\|\laps{t} u\|_{L^q(\R^n)}.
\]
By ellipticity of $\bar{A}$ and duality
{\small \[
\begin{aligned}
 \|\laps{t} v +c\Rz\otimes \Rz \laps{t} v\|_{L^p(\Omega_2)} \aleq &\|\bar{A} \brac{\laps{t} v +c\Rz\otimes \Rz \laps{t} v}\|_{L^p(\Omega_2)}\\
 \aleq&\sup_{\psi} \int_{\R^n} \langle\bar{A}(x)\brac{\laps{t} v +c\Rz\otimes \Rz \laps{t} v}(x),  \psi (x)\rangle dx,
\end{aligned}
 \]}
where the supremum is taken over all $\psi\in C_c^{\infty}(\Omega_2;\R^n)$. To finish the proof of the proposition, it suffices to prove that for any $\psi\in C_c^{\infty}(\Omega_2;\R^n)$
\[
\begin{split}
&\int_{\R^n} \langle\bar{A}(x)\brac{\laps{t} v +c\Rz\otimes \Rz \laps{t} v}(x),  \psi (x)\rangle dx \\
&\aleq \left(\|f_1\|_{L^{p}(\Omega_2)} + \|f_2\|_{L^{p}(\Omega_3)} + \|f_1\|_{L^{q} + (\R^n)} +\|\laps{t} v\|_{L^{r}(\Omega_4)} + \|\laps{t} v\|_{L^{q}(\R^{n})}\right)\|\psi\|_{L^{p'}(\Omega_3)}. 
\end{split}
\] 
To that end, pick $\eta_1, \eta_2 \in C_c^\infty(\Omega_3)$ with $\eta_1 \equiv 1$ in a neighborhood of $\Omega_2$, and $\eta_2 \equiv 1$ on the support of $\eta_1$. For any $\psi\in C_c^{\infty}(\Omega_2;\R^n)$, write $\psi = \laps{2s-t}\brac{\lapms{2s-t}\psi}$ and 
{\small \begin{equation}\label{psi-decomposed}
\begin{aligned}
\psi 
&=\laps{2s-t}\brac{\eta_1 \lapms{2s-t}\psi} +  \laps{2s-t}\brac{(1-\eta_1) \lapms{2s-t}\psi}\\
&=\laps{2s-t}\brac{\eta_1 \lapms{2s-t}\psi} +  \eta_2\laps{2s-t}\brac{(1-\eta_1) \lapms{2s-t}\psi} + (1-\eta_2)\laps{2s-t}\brac{(1-\eta_1) \lapms{2s-t}\psi}.
\end{aligned}
\end{equation}}
Then we have 
\[
\begin{split}
 &\int_{\R^n} \langle\bar{A}\brac{\laps{t} v +c\Rz\otimes \Rz \laps{t} v}(x), \psi(x)\rangle dx\\
 =&\int_{\R^n} \langle\bar{A}\brac{\laps{t} v +c\Rz\otimes \Rz \laps{t} v}(x),  \laps{2s-t}\brac{\eta_1 \lapms{2s-t}\psi}(x)\rangle dx\\
&+\int_{\R^n} \langle\bar{A}\brac{\laps{t} v +c\Rz\otimes \Rz \laps{t} v}(x), \eta_{2}(x) \laps{2s-t}\brac{(1-\eta_1) \lapms{2s-t}\psi}(x) \rangle dx\\&+\int_{\R^n} \langle\bar{A}\brac{\laps{t} v +c\Rz\otimes \Rz \laps{t} v}(x), (1-\eta_{2}) \laps{2s-t}\brac{(1-\eta_1) \lapms{2s-t}\psi}(x) \rangle dx\\
&=I + II + III.
 \end{split}
 \]
 We estimate each of the above integrals. To estimate $I$, we set $\varphi := \eta_1 \lapms{2s-t}\psi \in C_c^\infty(\Omega_3)$. Then we notice that $\varphi$ is an admissible test function in the equation \eqref{tested-weighted} and, thus,  
 we can use it in the equation  
\[
\begin{split}
 I &= \int_{\R^n} \langle f_1(z),\, \laps{2s-t} \varphi(z)\rangle\, dz + \int_{\R^n} \langle f_2(z),\, \varphi(z)\rangle\, dz\\
&= \int_{\R^n} \langle f_1(z),\, \psi(z)\rangle\, dz + \int_{\R^n} \langle f_2(z),\, \varphi(z)\rangle\, dz -\int_{\R^n} \langle f_1(z),\, \laps{2s-t} (1-\eta_1) I^{2s-t} \psi\rangle\, dz. 
  \end{split}
\]
where the latter is obtained  using the decomposition  \eqref{psi-decomposed}.
Now the first two terms can be estimates as follows:
\[
\begin{split}
\int_{\R^n} \langle f_1(z),\, \psi(z)\rangle\, dz + \int_{\R^n} \langle f_2(z),\, \varphi(z)\rangle\, dz &\leq \|f_1\|_{L^{p}(\Omega_2)} \|\psi\|_{L^{p'}(\Omega_2)}  + \|f_2\|_{L^{p}(\Omega_3)} \|\varphi\|_{L^{p'}(\Omega_3)}\\
&\aleq\left(\|f_1\|_{L^{p}(\Omega_2)} + \|f_2\|_{L^{p}(\Omega_3)}\right)\|\psi\|_{L^{p'}(\Omega_2)}  
\end{split}
\]
where we used Sobolev inequalities \eqref{eq:sob:loc1} and \eqref{eq:sob:loc2} and the fact that $\psi$ is compactly supported to estimates
\[
\|\varphi\|_{L^{p'}(\Omega_3)} \leq \|\lapms{2s-t}\psi\|_{L^{p'}(\Omega_3)} \aleq \|\psi\|_{L^{p'}(\Omega_2)}. 
\]
To estimate the last term of $I$, first we write it as  
\[
\begin{split}
\int_{\R^n} \langle f_1(z),\, \laps{2s-t} (1-\eta_1) I^{2s-t} \psi\rangle\, dz 
&= \int_{\R^n} \langle f_1(z),\, \eta_2\laps{2s-t} (1-\eta_1) I^{2s-t} \psi\rangle\, dz\\
&+ \int_{\R^n} \langle f_1(z),\, (1-\eta_2)\laps{2s-t} (1-\eta_1) I^{2s-t} \psi\rangle\, dz.
\end{split}
\]
Then while application of \eqref{prop2.4-partb} of  Lemma \ref{Prop2.4MSY} 
(or \cite[Proposition 2.4 part b)]{MSY21}) yields, 
\[
\begin{split}
\int_{\R^n} \langle f_1(z),\, \eta_2\laps{2s-t} ((1-\eta_1) I^{2s-t} \psi)\rangle\, dz &\leq \|f_{1}\|_{L^{p}(\Omega_3)} \|\laps{2s-t} (1-\eta_1) I^{2s-t} \psi\|_{L^{p'}(\Omega_3)} \\
&\aleq  \|f_{1}\|_{L^{p}(\Omega_3)} \|\psi\|_{L^{p'}(\Omega_2)} 
\end{split}
\]
and application of \eqref{prop2.4-parta} of  Lemma \ref{Prop2.4MSY} 
which holds for any $r>1$ implies that 
\[
\begin{split}
 \int_{\R^n}& \langle f_1(z),\, (1-\eta_2)\laps{2s-t} (1-\eta_1) I^{2s-t} \psi\rangle\, dz \\&\leq \|f_{1}\|_{L^{q}(\mathbb{R}^{n})} \|(1-\eta_2)\laps{2s-t} (1-\eta_1) I^{2s-t} \psi \|_{L^{q'}(\mathbb{R}^{n})}\\
 &\aleq \|f_{1}\|_{L^{q}(\mathbb{R}^{n})}\|\psi\|_{L^{p'}(\Omega_2)}.
\end{split}
\]
That finishes the estimate for $I$. To estimate $II$, we again apply \eqref{prop2.4-partb} of  Lemma \ref{Prop2.4MSY}  
to obtain 
\[
\begin{split}
II &\leq \int_{\R^n} |\bar{A}\brac{\laps{t} v +c\Rz\otimes \Rz \laps{t} v}(x)| |\eta_{2}(x) \laps{2s-t}\brac{(1-\eta_1) \lapms{2s-t}\psi}(x)| dx  \\
&\leq \|\bar{A}\brac{\laps{t} v +c\Rz\otimes \Rz \laps{t} v}\|_{L^{r}(\Omega_3)} \|\laps{2s-t}\brac{(1-\eta_1) \lapms{2s-t}\psi}\|_{L^{r'}(\Omega_3)}\\
&\aleq \|{\laps{t} v}\|_{L^{r}(\Omega_4)}\|\psi\|_{L^{p'}(\Omega_3)}.
\end{split}
\]
Finally, the estimate $III$ follows from \eqref{prop2.4-parta} of  Lemma \ref{Prop2.4MSY} 
as 
{\small \[
\begin{split}
III &\leq \int_{\R^n} |\bar{A}\brac{\laps{t} v +c\Rz\otimes \Rz \laps{t} v}(x)| |(1-\eta_{2})(x) \laps{2s-t}\brac{(1-\eta_1) \lapms{2s-t}\psi}(x)| dx  \\
&\aleq \|{\laps{t} v}\|_{L^{q}(\Omega_4)}\|\psi\|_{L^{p'}(\Omega_3)}.
\end{split}
\]}
That concludes the proof of the proposition. 

\end{proof}
We are now ready to state and prove the optimal regularity result for the weighted fractional Lam\'e equation given in \eqref{weighted-FLAME}. 
The result follows from \Cref{intermediate-regular-Klaps} by iterating the result on successive subdomains. We sketch its proof below. 
\begin{theorem}
\label{pr:ourtheorem1.7}
Let $s  \in (0,1)$, $t \in (0,2s)$ such that $2s-t<1$. Assume that for some $q \in (1,\infty)$, $\laps{t} u \in L^q(\R^n)$ is a distributional solution to
 \[
 \begin{split}
 & \int_{\R^d} \langle\bar{A}(z) (\laps{t} u +c\Rz\otimes \Rz \laps{t} u)(z), \laps{2s-t} \varphi(z)\rangle\, dz \\
  &= \int_{\R^n} \langle f_1(z)\,, \laps{2s-t} \varphi(z)\rangle\, dz + \int_{\R^n} \langle f_2(z)\,, \varphi(z)\rangle\, dz, \quad \forall \varphi \in C_c^\infty(\Omega).
  \end{split}
\]
Here $\bar{A}: \R^n \to \R$ is a positive, measurable, and bounded from above and below, i.e.
\[
\Lambda^{-1} \leq \bar{A}(z) \leq  \Lambda \quad \text{a.e. }x \in \R^n.
\]
Then for any $\Omega' \subsubset \Omega \subsubset \R^n$, $p\in (1, \infty)$, if $f_1,f_2 \in  L^q(\R^n) \cap L^p(\Omega)$,  then $\laps{t} u \in L^p(\Omega')$ with
\[
 \|\laps{t} u\|_{L^p(\Omega')} \aleq \|f_1\|_{L^p(\Omega)} + \|f_2\|_{L^p(\Omega)}+ \|f_1\|_{L^q(\R^n)} + \|\laps{t} u\|_{L^q(\R^n)}.
\]
If, in addition, $\bar{A}$ is $\gamma$-H\"older continuous uniformly, that is 
\[
\sup_{x, y \in \R^n}{|\bar{A}(x) - \bar{A}(y)|\over |x-y|^\gamma} \leq \Lambda,
\]
then for any $\beta\in (0, \min\{\gamma, 2s-t\}$ and any $\Omega'\subsubset\Omega\subsubset\R^n$
\[
\|\laps{{t+\beta}}u\|_{L^{p}(\Omega')} \aleq  \|\laps{t}u \|_{L^{q}(\R^n)} + \|\laps{\beta} f_1\|_{L^{q}(\R^n)} + \|f_{2}\|_{L^{q}(\Omega)}.  
\]

\end{theorem}
\begin{proof} 
Suppose that $p>q, $ and $2s-t <1$. We consider a sequence of pairs $(\Omega_i, p_i )$ for $i=1, 2, \cdots, L$ such that 
\[
\Omega_1 = \Omega', p_1 = p, \quad \Omega_i\subsubset \Omega_{i+1}, 
\] 
\[
p_{i+1} \in [q, p_i]\quad \text{such that } p_{i+1} > {n p_{i} \over n + (2s-t)p} > 1, 
\]
for some $L$ and $p_{L+1} = q$, $p_L = \bar{\mathfrak{q}}$, where $\bar{\mathfrak{q}}$ is obtained in part a) of \Cref{intermediate-regular-Klaps}. A finite $L$ depending on $t, s, p, q$ exists. We then apply part b) \Cref{intermediate-regular-Klaps} to obtain the inequality that 
\[
 \|\laps{t} u\|_{L^{p_{i}}(\Omega_i)} \aleq \sum_{j=1}^{2}\left(\|f_j\|_{L^{p_i}(\Omega_{i+1})} + \|f_j\|_{L^q(\R^n)}\right) + \|\laps{t} u\|_{L^{p_{i+1}}(\Omega_{i+1})} +\|\laps{t} u\|_{L^q(\R^n)}.
\]
We now iterate to get the desired inequality. 

The second part of the theorem can be proved in exactly the same was \cite[Proposion 4.2]{MSY21}.
\end{proof}

\section{The regularity theorem: Proof of Theorem~\ref{th:main}}\label{sec-regularity-proof}

\Cref{th:main} will be proved by an iteration argument that is explained in detail in \cite{MSY21}. In short, it follows the following steps. First, we obtain a localized  small incremental improvement  for a solution to a globally posed problem. Second, via a cutoff argument, extend the solution with locally improved regularity to be globally defined and also at the same time solve a globally posed problem. This extension is accompanied by essential controlled estimates. We now iterate and get a localized small improved regularity further increasing the regularity of the solution, and so on.   The localizing estimate can be done in exactly the same way as \cite[Theorem 5.1]{MSY21}. The only component missing is the ``small localized improvement'' that replaces \cite[Theorem 6.1]{MSY21}. In the remaining, we will only prove this missing regularity result and refer the execution of the iterative argument to \cite{MSY21}.  
\begin{theorem}\label{th:slighincrease}
Fix $s \in (0,1)$, $t \in [s,2s)$, $t <1$. For  given $\alpha\in (0, 1)$, $\lambda,\Lambda >0$, let $A\in \mathcal{A}(\alpha,\lambda, \Lambda)$. Suppose also that for any $2 \leq p < \infty$,  $u \in H^{s,2}(\R^n,\R^n) \cap H^{t,p}(\R^n,\R^n) \cap H^{t,2}(\R^n,\R^n)$ with $\supp u \subset \Omega \subsubset \R^n$ is a solution to
\begin{equation}\label{eq:slight:1}
\langle \mathcal{L}^{s}_{A} u, \varphi\rangle= \int_{\R^{n}} \langle f_1,  \laps{2s-t} \varphi\,\rangle dz 
 + \int_{\R^{n}} \langle f_2,  \varphi\, \rangle dz,\quad \forall \varphi \in C_c^\infty(\R^n).
\end{equation}
Then there exists $\bar{\eps} > 0$ such that if $r \in [p,p+\bar{\eps})$ and $f_1,f_2 \in L^r(\R^n)\cap L^p(\R^n),$ then
\[
 \|\laps{t} u\|_{L^r(\Omega)} \aleq \sum_{i=1}^2\|f_i\|_{L^r(\R^n)} + \|f_i\|_{L^p(\R^n)}+ \|\laps{t} u\|_{L^p(\R^n)}.
\]
In addition, if $\beta \in [0,\bar{\eps}]$, $\laps{\beta} f_1 \in L^p(\R^n)$,  and $f_1,f_2 \in L^p(\R^n)$,  then $\laps{t+\beta} u \in L^p_{loc}(\R^n)$ and 
\begin{equation}\label{beta-diff}
 \|\laps{t+\beta} u\|_{L^p(\Omega)} \aleq \|\laps{\beta}f_1\|_{L^p(\R^n)} + \|f_1\|_{L^p(\R^n)}+\|f_2\|_{L^p(\R^n)}+ \|\laps{t} u\|_{L^p(\R^n)}.
\end{equation}
Here, $\bar{\eps} > 0$ is uniform in the following sense: $\bar{\eps}$ depends only on $\alpha$ and the number $\theta \in (0,1)$  which is such that
\[
 \theta < s,t,2s-t < 1-\theta,\,\,\text{and   \,\,$ 2 \leq p < \frac{1}{\theta}.$}
\]
\end{theorem}
\begin{proof}
We proceed very similar to \cite{MSY21}, by reformulating the system of equations to the weighted fractional Lam\'e system studied in \Cref{intermediate-regular-Klaps} -- up to the commutators introduced in \Cref{s:commies}. Set
\[
F[\varphi] := \int_{\R^{n}} \langle f_1\,, \laps{2s-t} \varphi\,\rangle dz 
 + \int_{\R^{n}} \langle f_2\,, \varphi\,r\rangle dz.
\]
The for $t\in [s, 2s)$, if $u$ solves \eqref{eq:slight:1}, then recalling  the decomposition \eqref{2nd-div}, we have  up to a constant multiple 
\begin{equation}\label{weighted-lame-red}
\langle\bar{\mathfrak{L}}_{A_D}^{t,2s-t} u, \varphi\rangle =F[\varphi] -\mathcal{D}_{1}^{s}(u,\varphi) - \mathcal{D}^{t,2s-t}_{2}(u,\varphi), \quad \text{for all   $\varphi\in C_{c}(\R^{n}; \R^{n})$. } 
\end{equation}
where  the linear operator $\bar{\mathfrak{L}}_{A_D}^{t,2s -t}$ is   the weighted fractional Lam\'e operator introduced in  \eqref{fractional-Lame} with $A_{D}(z) = A(z, z)$ is bounded from below and above by positive constants. The  functionals $\mathcal{D}_{1}^{s}$ and $\mathcal{D}_{2}^{s_1, s_2}$ are as defined in \Cref{s:commies}. We now define the two operators
\[
T_{1}[\varphi] = \mathcal{D}_{1}^{s}(u,\varphi),\quad \text{and} \quad T_2[\varphi] = \mathcal{D}^{s_1,s_2}_{2}(u,\varphi)
\]
which are linear in  $\varphi\in C_{c}^{\infty}(\mathbb{R}^{n})$.  
Given $\theta$ as in the theorem, we can choose $\epsilon$ sufficiently small so that if we take $\sigma =8\epsilon$ we have 
\[
{np'\over n+ \sigma p'} \in (1, \infty) \quad\text{for all $p\in (2, {1\over \theta})$}
\]
and that \Cref{pr:prop11} and \Cref{pr:prop12} hold. Applying \Cref{pr:prop11} with this $\sigma$ and Sobolev inequalities Lemma \ref{S-I} we see that for any $\beta\in [0, \epsilon]$
\[
\begin{split}
|T_1[\varphi]| &\aleq \int_{\mathbb{R}^n} |\laps{t}u|(x) \lapms{\sigma-\epsilon} |\laps{2s-t-\epsilon}\varphi|(x)\,dx\\
&\aleq\|\laps{t}u\|_{L^{p}}\|\laps{2s-t-\epsilon}\varphi\|_{L^{{np'\over n + (\sigma-\epsilon)p'}}(\R^{n})}\\
&\aleq\|\laps{t}u\|_{L^{p}}\|\laps{2s-t-\epsilon}\varphi\|_{L^{{np'\over n + (\sigma-\beta)p'}}(\R^{n})}.
\end{split}
\]As a consequence, we have that for any $\beta\in [0, \epsilon]$
\[
T \in \left( \dot{H}^{2s-t-\beta, {np' \over n + (\sigma -\beta)p'}} (\R^n)\right)^\ast.
\]
By representation of the dual elements \cite[Proposition 2.2]{MSY21}, there exists $g^1_\beta \in L^{{np \over n-(\sigma -\beta)p} }(\R^n, \R^n)$ such that 
\[
T_1[\varphi] = \int_{\R^n} \langle g^1_{\beta}(x), \laps{2s-t-\beta}\varphi(x)\rangle \,dx.  
\]
Similarly, applying \Cref{pr:prop12} and repeating the above calculation for $T_2$, for any $\beta\in [0, \epsilon]$, and $\epsilon < \alpha$, we can get from the representation of dual elements that a vector field $g^2_\beta \in L^{{np \over n-(\sigma -\beta)p} }(\R^n, \R^n)$ such that
\[
T_2[\varphi] = \int_{\R^n} \langle g^2_{\beta}(x), \laps{2s-t-\beta}\varphi(x)\rangle \,dx. 
\]
After denoting $g_{\beta} := g^1_{\beta} + g^2_{\beta}$, we can now rewrite \eqref{weighted-lame-red} as
\[
\langle\bar{\mathfrak{L}}_{A_D}^{t,2s-t} u, \varphi\rangle = \int_{\R^n}\langle \laps{\beta} f_1 + g_{\beta}, \laps{2s-t-\beta} \varphi\rangle \,dx + \int_{\R^{n}} \langle f_2,\varphi \rangle  
\]
for all $\beta\in[0, \epsilon]$ and $\varphi\in C_c(\mathbb{R}^{n}, \R^{n})$. 

Now if $\beta=0$, then we may apply \Cref{pr:ourtheorem1.7} to conclude that for any $\Omega\subsubset \R^n$ and $r\in[p, {np \over n-\sigma p}]$ we have 
\[
\|\laps{t} u\|_{L^{r}(\Omega)}\aleq \sum_{i=1}^{2}\|f_i\|_{L^{r}(\R^n)} + \|f_i\|_{L^{p}(\R^n)}+\|\laps{t} u\|_{L^{p}(\R^n)}.
\]
We notice that there exists an $\bar{\epsilon}>0$  such that ${np\over n-\sigma p } \geq p + \bar{\epsilon}$ for all $p\in [0, {2\over \theta}]$. 

Also, if $\beta\in (0, \epsilon)$, since $A_D$ is $\alpha$-H\"older ocntinous uniformly, we may apply the second part of \Cref{pr:ourtheorem1.7} to obtain \eqref{beta-diff}. 
\end{proof}

\section*{Acknowledgement}
The authors acknowledge funding as follows
\begin{itemize}
\item TaMen: National Science Foundation (NSF), grant no DMS-2206252.
\item AdSee: Thailand Science Research and Innovation Fundamental Fund fiscal year 2024
\item ArSch: National Science Foundation (NSF), NSF Career DMS-2044898
\item SaYee: Thammasat Postdoctoral Fellowship
\end{itemize}
The research that lead to this work was partially carried out while ArSch was visiting Chulalongkorn University. ArSch is a Humboldt fellow.

\bibliographystyle{abbrv}
\bibliography{bib}

\end{document}